\documentclass[11pt]{article}
\usepackage[T1]{fontenc}
\usepackage{tocbibind}
\usepackage{amsmath,amsthm,amssymb}
\usepackage{fancyhdr}
\usepackage[british]{babel}
\usepackage{geometry,mathtools}
\usepackage{enumitem}
\usepackage{algpseudocode}
\usepackage{dsfont}
\usepackage{centernot}
\usepackage{xstring}
\usepackage{colortbl}
\usepackage[section]{algorithm}
\usepackage{graphicx}
\usepackage[font={footnotesize}]{caption}
\usepackage[usenames,dvipsnames,table]{xcolor}
\usepackage{pstricks,tikz}
\usepackage[h]{esvect}
\usepackage[
bookmarksopen=true,
bookmarksopenlevel=1,
colorlinks=true,
linkcolor=darkblue,
linktoc=page,
citecolor=darkblue,
]{hyperref}
\usepackage{bbm}

\numberwithin{equation}{section}

\definecolor{darkblue}{rgb}{0,0,0.5}

\newdimen\margin
\def\textno#1&#2\par{
	\margin=\hsize
	\advance\margin by -4\parindent
	\setbox1=\hbox{\sl#1}
	\ifdim\wd1 < \margin
	$$\box1\eqno#2$$
	\else
	\bigbreak
	\hbox to \hsize{\indent$\vcenter{\advance\hsize by -3\parindent
			\it\noindent#1}\hfil#2$}
	\bigbreak
	\fi}

\newtheorem{theorem}[algorithm]{Theorem}

\newtheorem{lemma}[algorithm]{Lemma}

\theoremstyle{definition}

\newtheorem{defin}[algorithm]{Definition}

\def\lateproof#1{\removelastskip\penalty55\medskip\noindent\begin{stepenv}\end{stepenv}{\bf Proof of #1. }} 

\def\noproof{{\unskip\nobreak\hfill\penalty50\hskip2em\hbox{}\nobreak\hfill%
		$\square$\parfillskip=0pt\finalhyphendemerits=0\par}\goodbreak}
\def\endproof{\noproof\bigskip}

\newcounter{stepenv}
\newenvironment{stepenv}[1][]{\refstepcounter{stepenv}}{}

\newcounter{step}[stepenv]

\newcounter{substep}[step]
\renewcommand{\thesubstep}{\thestep.\arabic{substep}}

\newcounter{claim}[stepenv]

\newcommand{\eul}{{e}}
\newcommand{\ind}{\mathrm{ind}}

\def\sm{\setminus}

\def\COMMENT#1{}
\def\TASK#1{}
\let\TASK=\footnote             

\begin{document}

	\title{Short proofs for long induced paths}

	\author{
		Nemanja Dragani\'c \thanks{Department of Mathematics, ETH, 8092 Z\"urich, Switzerland.
			Email: \href{mailto:nemanja.draganic@math.ethz.ch}{\nolinkurl{nemanja.draganic@math.ethz.ch}}. Research supported in part by SNSF grant 200021\_196965.
		}		
		\and
		Stefan Glock \thanks{Institute for Theoretical Studies, ETH, 8092 Z\"urich, Switzerland.
			Email: \href{mailto:dr.stefan.glock@gmail.com}{\nolinkurl{dr.stefan.glock@gmail.com}}. Research supported by Dr.~Max R\"ossler, the Walter Haefner Foundation and the ETH Z\"urich Foundation.}
		\and
		Michael Krivelevich \thanks{School of Mathematical Sciences, Raymond and Beverly Sackler Faculty of Exact Sciences, Tel Aviv University, Tel Aviv, 6997801, Israel. Email: \href{mailto:krivelev@tauex.tau.ac.il}{\nolinkurl{krivelev@tauex.tau.ac.il}}. Research supported in part by USA-Israel BSF grant 2018267, and by ISF grant 1261/17.}
	}
	
	\date{}
	
	\maketitle

	\begin{abstract} 
		We present a modification of the Depth first search algorithm, suited for finding long induced paths. We use it to give simple proofs of the following results. We show that the induced size-Ramsey number of paths satisfies $\hat{R}_{\ind}(P_n)\leq 5\cdot 10^7n$, thus giving an explicit constant in the linear bound, improving the previous bound with a large constant from a regularity lemma argument by Haxell, Kohayakawa and {\L}uczak. We also provide a bound for the $k$-colour version, showing that $\hat{R}_{\ind}^k(P_n)=O(k^3\log^4k)n$. Finally, we present a new short proof of the fact that the binomial random graph in the supercritical regime, $G(n,\frac{1+\varepsilon}{n})$, contains typically an induced path of length $\Theta(\varepsilon^2) n$.
	\end{abstract}

	\section{Introduction}
	In this article, we give short proofs for two well known problems regarding finding long induced paths in random graphs.
	
	The first problem concerns the induced size-Ramsey number of paths. For a graph $H$ we define the \emph{$k$-colour induced size-Ramsey number} of $H$, denoted by $\hat{R}_{\ind}^k(H)$, as the smallest number $m$ such that there exists a graph $G$ on $m$ edges such that for every $k$-colouring of the edges of $G$, there is a monochromatic copy of $H$ which is an induced subgraph of~$G$. In 1987, Graham and R\"{o}dl~\cite{graham1987numbers} asked if the induced size-Ramsey numbers of paths $P_n$ are linear in~$n$ (for any fixed number of colours). This was confirmed by Haxell, Kohayakawa and {\L}uczak~\cite{HKL:95}, who showed that $\hat{R}_{\ind}^k(P_n)\leq c_k n$ for every fixed~$k$.
	Their proof is quite technical and is based on the regularity lemma, hence the derived constants $c_k$ are astronomically large. We revisit this problem and give a short and rather simple proof of the fact that the induced size-Ramsey numbers of paths are linear. Moreover, we obtain an explicit absolute constant for the $2$-colour version, and give a bound polynomial in $k$ for $c_k$ in the general case, for any fixed number of colours~$k$.
	
	\begin{theorem}\label{thm:ramsey}
		$\hat{R}_{\ind}(P_n)\leq 5\cdot 10^{7}n$ for all large enough $n$.
	\end{theorem}
	\begin{theorem}\label{thm:multicolour}
		$\hat{R}^k_{\ind}(P_n)=O(k^3\log^4k)n$.
	\end{theorem}
	
	The second classical problem we address is finding a linear sized induced path in a binomial random graph $G(n,p)$ in the supercritical regime, i.e.~when $p=\frac{1+\varepsilon}{n}$ for a sufficiently small positive constant~$\varepsilon$. 
	We give a short alternative proof of the following result, originally due to Suen~\cite{suen:92}.
	
	\begin{theorem}\label{thm:supercritical}
		There exists a constant $\varepsilon_0>0$ such that for all positive $\varepsilon<\varepsilon_0$, the random graph $G\sim G(n,\frac{1+\varepsilon}{n})$ with high probability\footnote{that is, with probability tending to $1$ as $n\to\infty$} (whp) contains an induced path of length $\frac{\varepsilon^2n}{5}$.
	\end{theorem}
	We remark that the dependency on $\varepsilon$ is optimal: it is known that even the length of a (not necessarily induced) path is whp $O(\varepsilon^2 )n$ (see, e.g., \cite{JLT:11}). Let us also note that Suen's result is slightly stronger in the sense that $1/5$ can be replaced by any constant smaller than~$1$.
	
	We give some background on those two problems and prove the results above in Sections~\ref{sec:Ramsey} and~\ref{sec:supercritical}.
	The common tool which we use in our proofs is a modified version of the Depth first search (DFS) graph search algorithm. By nature, DFS is very suitable for finding long paths in graphs.
	Our version is tailored for finding long \emph{induced} paths, specifically in graphs with certain local density conditions, hence it comes in handy for applications in random graphs. 
	We will first present the proof of Theorem~\ref{thm:supercritical}, where the DFS algorithm is used directly in the random graph $G(n,p)$. Subsequently, in the proofs of Theorems~\ref{thm:ramsey} and~\ref{thm:multicolour}, we apply it in a monochromatic subgraph of a random graph.
	Throughout, we treat large numbers like integers whenever this has no effect on the argument.

	\section{DFS for induced paths}\label{sec:DFS}
	In the standard DFS algorithm, we explore the vertices one by one, always following one branch as far as possible, before we start backtracking.
	Given a graph $G$, the idea is to keep track of three sets of vertices $U,T,S$, where $T$ is the set of unexplored vertices, $S$ is the set of vertices whose exploration is complete, and the remaining vertices $U$ are kept in a stack. At every step we look at the vertex $u$ which is the last one added to $U$, and try to find a neighbour $t$ of $u$ in $T$. If we succeed, we move $t$ to $U$, and if not, we move $u$ to $S$. It is easy to see that the vertices in $U$ contain a spanning path in $G[U]$. In our modified version of the DFS algorithm (described below), after finding $t$, we also check if $t$ has any other neighbours in $U$ except $u$, and if so, we move $t$ to $S$. This ensures that $U$ always spans an induced path in $G$, and makes the algorithm suitable for finding long induced paths in sparse expanders. 
	
	More precisely, our goal is the following. Given two graphs $G',G$ on the same vertex set and with $G'\subseteq G$, we want to find a long induced path in $G$, whose edges are all in $G'$. When $G'=G$ this boils down to finding a long induced path in $G'$, but with the Ramsey question mentioned before in mind, it will be convenient for us to formulate the algorithm with two input graphs, so that in this specific instance, $G'$ will be a monochromatic subgraph of the coloured host graph $G$.
	In applications we usually run the algorithm up to a certain stage, and by analyzing it we conclude that the input graphs contain a suitable induced path. 
	
	The algorithm is a graph search algorithm which visits all the vertices in the following manner. As input, it receives graphs $G'=(V,E')$ and $G=(V,E)$ with $E'\subseteq E$, and an ordering $\pi$ of $V$. The algorithm maintains four sets of vertices $U,T,S_1$ and $S_2$. The set $T$ is the set of unvisited vertices, $S_1$ and $S_2$ are the sets of discarded vertices, while $U=V\sm(T\cup S_1\cup S_2)$ is the set of remaining vertices which are kept in a stack (the last vertex to enter $U$ is the first to leave). At every stage of the algorithm $U$ will induce a path in $G$ with all edges belonging also to $G'$. In the beginning we set $S_1=S_2=U=\emptyset$ and $T=V$, and we stop when $U=T=\emptyset$.
	The algorithm is carried out in rounds, and in each round we proceed as follows. 
	
	\bigskip
	
	\noindent \textbf{Beginning of round}
	
	\begin{enumerate}
		\item\label{step1} If $U$ is empty, we take the first vertex in $T$ according to $\pi$, remove it from $T$ and push it to $U$. 
		\item\label{step2} Otherwise, let $u$ be the vertex on the top of the stack in $U$. Now we query $T$ for vertices $t\in T$ such that $(u,t)$ is an edge in $E'$, by scanning $T$ according to the ordering $\pi$. We have one of the following scenarios, given by Steps~\ref{step3} and~\ref{step4}.
		
		\item\label{step3} If an appropriate $t$ is found, we query $U\sm\{u\}$ for vertices 
		$u'\in U\sm\{u\}$ such that $(t, u')\in E$ by scanning $U \sm \{u\}$ according to the ordering $\pi$.
		\begin{enumerate}
			\item If all the answers are negative, remove $t$ from $T$ and push it to $U$.
			\item If we get at least one positive answer, remove $t$ from $T$ and add it to $S_2$.
		\end{enumerate}
		\item\label{step4} If no such $t$ is found, remove $u$ from $U$ 
		and add it to $S_1$.
	\end{enumerate}
	
	\noindent \textbf{End of round}\\
	\bigskip
	
	In order to explore all pairs of vertices in the graph, for technical reasons, we also query all the pairs in $V$ which were not queried before for being in~$G'$ (in the first paragraph of the proof of Theorem~\ref{thm:supercritical} it becomes apparent why we want to query all pairs). This completes the algorithm.
	
	The following properties of the algorithm will play an important role in analyzing it in the sections that follow. 
	
	\begin{enumerate}[label= (\Alph*)]
		\item\label{fact:bipartite} At every point all pairs between $S_1$ and $T$ have been queried, and none of them is in $E'$.
		\item\label{fact:move from T} 
		Every time we enter Step~\ref{step3}, the size of $U\cup S_1\cup S_2$ increases by $1$, and it never decreases.
		\item\label{fact:2k edges} At every point the number of edges in $G[U\cup S_1\cup S_2]$ is at least $2|S_2|$. 
		\item\label{fact:induced path}  At every point of the algorithm, $U$ induces a path in $G$ with all its edges being also in $G'$.
	\end{enumerate}
	Properties~\ref{fact:bipartite},~\ref{fact:move from T} and~\ref{fact:induced path} hold for immediate reasons, while~\ref{fact:2k edges} holds since every vertex which lands in $S_2$ has at least two neighbours in $G$ in the current $U$, and all the vertices in this current $U$ either stay at $U$ or go to $S_1$.
	
	Now we give a result which shows that given two graphs $G'\subseteq G$, if $G$ satisfies a local density condition, and $G'$ has a certain expansion property, then $G'$ contains a long path induced in $G$.
	It follows from analyzing our modified DFS algorithm, and it will be used to prove our Ramsey results. 
	
	Given a graph $G$ and a subset of vertices $S$, we denote by $N_{G}(S)$ the \emph{external} neighbourhood of $S$, that is, the set of vertices outside $S$ which have a neighbour in~$S$.
	\begin{theorem}\label{thm:local considerations}
		Let $G,G'$ be graphs on the same vertex set with $G'\subseteq G$, and let $s_1,s_2$ and $\ell$ be positive integers such that for every set of vertices $S$ the following hold:
		\begin{itemize}
			\item If $|S|< s_1+s_2+\ell$, then $|E(G[S])|< 2s_2$;
			\item If $|S|=s_1$, then $|N_{G'}(S)|\geq s_2+\ell$.
		\end{itemize}
		If $|V(G)|\geq \ell+s_1+s_2$, then $G'$ contains a path of length $\ell$ which is an induced path in $G$.  
	\end{theorem}
	
	\begin{proof}
		In order to find the path, we run the algorithm described above with input graphs $G'$ and $G$, with an arbitrary ordering of their vertices $\pi$. Let us show that at the first point when either $|S_1|=s_1$ or $|S_2|=s_2$, $U$ induces a path of length~$\ell$ (observe that such a point exists by the lower bound on $|V(G)|$). By \ref{fact:induced path}, $U$ always induces a path in $G$ and all edges in the path are in $G'$; so this would give us precisely the path from the statement.
		Suppose for the sake of contradiction that $|U|\leq \ell$.

		First, assume that $|S_2|=s_2$ (but $|S_1|<s_1$) at the observed time. Then, by Property~\ref{fact:2k edges} it holds that $|E(G[U\cup S_1\cup S_2])|\geq 2|S_2|= 2s_2$, a contradiction with $|U\cup S_1\cup S_2|< s_1+s_2+\ell$ and our first assumption (on $G$). 
		
		Now, if actually $|S_1|=s_1$ (but $|S_2|<s_2$), by our second assumption (on $G'$) we have that $|N_{G'}(S_1)|\geq s_2+\ell$. But note that by Property~\ref{fact:bipartite} of the algorithm it holds that $N_{G'}(S_1)\subseteq S_2\cup U$ and hence $|N_{G'}(S_1)|<s_2+\ell$, a contradiction. This completes the proof.
	\end{proof}

	\section{Long induced paths in the supercritical regime}\label{sec:Ramsey}
	
	In this section, we prove Theorem~\ref{thm:supercritical}.
	Determining the order of a largest induced path/tree in a random graph is a well-known problem with a long history~\cite{DS:18,EP:83,fernandez-de-la-Vega:86,fernandez-de-la-Vega:96,FJ:87b,FJ:87a,KR:87,luczak:93,LP:88,suen:92}.
	Frieze and Jackson~\cite{FJ:87b} showed that for every sufficiently large $d$, there exists a constant $\alpha(d)>0$ such that whp the random graph $G(n,d/n)$ contains an induced path of length $\alpha(d)n$. {\L}uczak~\cite{luczak:93} and independently Suen~\cite{suen:92} showed that one can take $\alpha(d)\sim \frac{\log d}{d}$ as $d\to \infty$. This is optimal up to a factor $2$, as can be seen by a simple first moment calculation. Recently, the authors~\cite{DGK:21} obtained this ``missing'' factor in the lower bound, thus showing that whp the length of a longest induced path is asymptotically $\frac{2n}{d}\log d$.
	
	Here, we consider the case when $d$ is close to~$1$ (the so-called supercritical regime). {\L}uczak~\cite{luczak:93} and Suen~\cite{suen:92} also showed that for any constant $d>1$ whp there is an induced path of linear length, thus answering a question of Frieze and Jackson~\cite{FJ:87b}. In particular, Suen~\cite{suen:92} showed that one can take $\alpha(d)$ to be any constant smaller than $d^{-1}\int_{1}^d \frac{1-y(\xi)}{\xi}d\xi$, where $y(\xi)$ is the smallest positive root of $y=\eul^{\xi(y-1)}$. From this, one can derive Theorem~\ref{thm:supercritical}. 
	Our goal here is to present a simple proof of this result.
	Suen's proof is also based on a version of the DFS algorithm; in particular, he uses it to find large $m$-ary trees, and then he shows that the depth of one of the trees is large enough to guarantee a long path. Our version of the algorithm, combined with local density considerations, makes the analysis shorter and more straightforward. 
	
	We will need the following (rather standard) definition, which helps us quantify how far the components of a graph are from being trees. 
	\begin{defin}
		For a connected graph $G$, define the \emph{excess} of $G$ as $exc(G)=|E(G)|-|V(G)|+1$. If $G$ has more than one connected component then let $exc(G)$ be the sum of the excesses of each of its components. 
	\end{defin}
	The excess of a random graph in the supercritical regime typically comes overwhelmingly from the excess of its giant component, while the typical size of the giant component in terms of number of edges and number of vertices is well understood. We will use the following lemma (see, for example, Theorems 2.14 and 2.18 in \cite{frieze2016introduction}, and set $c=1+\varepsilon$, for small enough $\varepsilon$).
	\begin{lemma}
		There exists a constant $\varepsilon_0>0$ such that for all positive $\varepsilon<\varepsilon_0$, for the random graph $G\sim G(n,\frac{1+\varepsilon}{n})$ it holds whp that $exc(G)\leq \varepsilon^3n$. 
	\end{lemma}
	
	We are now ready to prove Theorem~\ref{thm:supercritical}. The argument follows closely that of Krivelevich and Sudakov~\cite{KS:13} in the non-induced case. One key idea is to construct the random graph ``on the fly'' while the DFS algorithm is executed. The source of randomness is a sequence of independent Bernoulli random variables which is used to answer the queries made by the algorithm. We use the same notation as in Section~\ref{sec:DFS}.
	
	\lateproof{Theorem~\ref{thm:supercritical}}
	We will run the algorithm defined in the previous section with $G'=G$ and an arbitrary ordering of the vertices $\pi$. We feed the algorithm with a sequence of i.i.d.~random variables $\{X_i\}_{i\in N}$ which follow a Bernoulli distribution with mean $p=\frac{1+\varepsilon}{n}$, where $N=\binom{n}{2}$, so that the $i$-th \emph{new query} of the algorithm is answered positively when $X_i=1$, and otherwise negatively. By new query, we mean a query made to a pair which has not yet been queried before (as in the third step we might query an already exposed pair, and there we just take its previous answer).
	Therefore, the explored graph obviously follows the distribution of $G(n,p)$, so our problem boils down to studying the properties of the random sequence $\{X_i\}_{i\in N}$.
	
	First, let us show that the number of vertices in $S_2$ is always at most the excess of $G$. Each vertex in $S_2$, before leaving $T$, was adjacent to at least two vertices on the path induced by $U$, so it contributes at least one to the excess of $G$, as it adds one vertex but at least two edges to its own connected component. Crucially, notice that the sets of contributing edges for each vertex $S_2$ are disjoint, as the at least two neighbouring vertices in $U$ are never added to $S_2$.
	Since whp $exc(G)\le \varepsilon^3n$, we have the same bound on $|S_2|$ whp. 
	
	Suppose for the sake of contradiction that we always have $|U|\leq \frac{\varepsilon^2n}{5}$.
	For the analysis of the algorithm, we will focus on the pairs $(u,t)$ which were queried when $u$ was in $U$ and $t$ was in $T$, i.e.~the pairs queried in Step~\ref{step2} of the algorithm.
	Let us show that whp at the point when we queried $N_0:=\frac{\varepsilon n^2}{2}$ pairs of this type, then $U$ is of size at least $\frac{\varepsilon^2 n}{5}+1$, which would mean we are done by~\ref{fact:induced path}. Observe that we can assume that at some point we queried $N_0$ pairs of the mentioned type; indeed, when say $|T|=n/2$, by~\ref{fact:bipartite},  we queried at least $|T||S_1|=\frac{n}{2}(\frac{n}{2}-|S_2|-|U|)>\frac{n^2}{8}$ such pairs.
	
	Now, we observe that when we have queried $N_0$ pairs of the mentioned type, then $|S_1\cup S_2|<n/3$; if this is not the case, then at some point before we must have had $|S_1\cup S_2|=n/3$. 
	Since $|T|=n-|S_1|-|S_2|-|U|>n/2$, by~\ref{fact:bipartite} we have queried more than $|T||S_1|>|T|(n/3-\varepsilon^3 n)>n^2/10$ pairs of the observed type, which is larger than $N_0$, a contradiction. So $|S_1\cup S_2|<n/3$.
	
	When we queried precisely $N_0$ of our pairs (in Step~\ref{step2}), the expected number of positive answers among them is $\frac{\varepsilon(1+\varepsilon)n}{2}$, hence, using Chernoff bounds we whp get at least $\frac{\varepsilon(1+\varepsilon)n}{2}-n^{2/3}$ edges among the queried pairs, and hence at least this many vertices in $U\cup S_1\cup S_2$, thanks to Property~\ref{fact:move from T}. Hence, we also have $|S_1|\geq \frac{\varepsilon\left(1+\varepsilon\right)n}{2}-n^{2/3}-\frac{\varepsilon^2 n}{5}-\varepsilon^3n$.

	By~\ref{fact:bipartite} we have at least $|S_1||T|=|S_1|(n-|S_1|-|S_2|-|U|)$ queried pairs of the observed type, so we have:
	\begin{align*}
		N_0&\geq|S_1|\left(n-\varepsilon^3 n-\frac{\varepsilon^2 n}{5}-|S_1|\right)\\
		&\geq  \left(\frac{\varepsilon\left(1+\varepsilon\right)n}{2}-n^{2/3}-\frac{\varepsilon^2 n}{5}-\varepsilon^3n\right)\left(n-\frac{\varepsilon\left(1+\varepsilon\right)n}{2}+n^{2/3}\right)\\
		&> \frac{\varepsilon n^2}{2}+\frac{\varepsilon^2n^2}{20}-O(\varepsilon^3)n^2
	\end{align*}
	(where the second inequality uses $\frac{\varepsilon\left(1+\varepsilon\right)n}{2}-n^{2/3}-\frac{\varepsilon^2 n}{5}-\varepsilon^3n\leq |S_1|<n/3$, so the product grows with $|S_1|$), contradicting the assumption on $N_0$ for all small enough $\varepsilon>0$, which completes the proof.
	\endproof{}

	\section{Induced size-Ramsey number of paths}\label{sec:supercritical}
	
	The size-Ramsey number of $H$, denoted by $\hat{R}(H)$, is  the smallest number $m$ such that there exists a graph $G$ on $m$ edges with the property that for every 2-colouring of the edges of $G$, there is a monochromatic copy of $H$ in $G$. This notion was introduced by  Erd\H{o}s, Faudree,
	Rousseau and Schelp \cite{erdHos1978size}, and over the past few decades there has been a lot of research devoted to studying this and other related Ramsey functions. One of the classical problems posed by Erd\H{o}s was to determine the order of magnitude of $\hat{R}(P_n)$, and he actually conjectured that $\frac{\hat{R}(P_n)}{n}\rightarrow \infty$, which was disproved by Beck~\cite{beck1983size} who showed $\hat{R}(P_n)=O(n)$. Since then, there has been a series of papers concerned with giving more precise bounds on $\hat{R}(P_n)$; for lower bounds see \cite{bal2019new,beck1983size,B2001,dudek2017some}, and for upper bounds see \cite{beck1983size,B2001,dudek2015alternative,dudek2017some,letzter2016path}. The current records for lower and upper bounds are given by  Bal and DeBiasio~\cite{bal2019new}, and by Dudek and Pra{\l}at~\cite{dudek2017some}, respectively:
	$$
	(3.75+o(1))n\leq \hat{R}(P_n)\leq 74n.
	$$
	
	For the $k$-colour version of the size-Ramsey number of paths, almost tight asymptotic bounds are known in terms of $k$ \cite{dudek2015alternative,dudek2017some,krivelevich2019expanders,krivelevich2019long}:
	$$\Omega(k^2) n\leq \hat{R}^k(P_n)\leq O(k^2\log k)n.$$

	Concerning the induced size-Ramsey number of paths, Haxell, Kohayakawa and {\L}uczak~\cite{HKL:95} showed that $\hat{R}^k_{\ind}(P_n)$ is linear for any fixed $k$, but no reasonably small constant can be extracted from their proof even if $k=2$,
	as it relies on the regularity lemma.\footnote{On the other hand, their argument actually gives the stronger result that $\hat{R}^k_{\ind}(C_n)=O(n)$ for any fixed $k$.} We improve upon this considerably, showing that $\hat{R}_{\ind}^2(P_n)\leq 5\cdot 10^7n$ and $\hat{R}_{\ind}^k(P_n)\leq O(k^3\log^4 k)n$.
	As in previous proofs, our ``host graph'' will be a sparse random graph $G(n,c/n)$, where $c$ is a sufficiently large constant. We have already seen in the last section that whp there is an induced path of linear length. The additional challenge here is to guarantee such a path even if an adversary may delete half of the edges, say. 
	Fortunately, the DFS algorithm presented in Section~\ref{sec:DFS} is very robust and does not require the full randomness of the host graph, but performs well in ``locally sparse'' graphs with a mild expansion property (cf.~Theorem~\ref{thm:local considerations}). After a simple cleaning step, we can always guarantee such a pseudorandom graph in the densest colour class. Hence, our results are density-type results, i.e.~we prove that a subset of edges forming an appropriate percentage of the whole graph contains a long path induced in the host graph.

	\subsection{The two-colour result}
	We first show a simple lemma which collects several useful properties of a random graph with parameters tailored for the proof of Theorem~\ref{thm:ramsey}.
	
	\begin{lemma}\label{auxiliary}
		Let $G\sim G(n,64/n)$. Then $G$ has the following properties whp.
		\begin{enumerate}
			\item Every vertex set $S$ of size at most $\frac{196n}{10^7}$ spans less than $\frac{12}{7}|S|$ edges.
			\item Every two disjoint vertex sets $S,T$ of sizes $|S|=\frac{21n}{10^7}$ and $|T|\le \frac{175n}{10^7}$ satisfy $e(S,T)<\frac{95}{7}|S|$.
			\item $G$ has $(1+o(1))32n$ edges and $\Theta(n)$ isolated vertices.
		\end{enumerate}
	\end{lemma}
	
	\begin{proof}
		\begin{enumerate}
			\item 
			Let $p=64/n$ and let $t=\frac{196n}{10^7}$. We bound the probability of the existence of a set $S$ of size at most $t$ which spans at least $\frac{12}{7}|S|$ edges by using the following simple union bound
			\begin{align*}
				\sum_{i\le t }\binom{n}{i}\binom{\binom{i}{2}}{\frac{12}{7}i}\cdot p^{\frac{12}{7}i}\le
				\sum_{i\le t}\left(\frac{en}{i}\right)^i\cdot \left(\frac{7eip}{24}\right)^{\frac{12}{7}i}
				&=\sum_{i\le t}\left[\frac{en}{i}\cdot\left(\frac{7\cdot 64ei}{24n}\right)^{\frac{12}{7}}\right]^i\\
				&\le \sum_{i\le t}\left[2280 \left(\frac{i}{n}\right)^\frac{5}{7}\right]^i
				\,.
			\end{align*}
			The part of the sum for $i<\sqrt{n}$ is dominated by $\sum_{i\le \sqrt{n}}\left(\frac{1}{n^{1/4}}\right)^i\rightarrow 0$. Otherwise, the $i$-th summand is bounded by $\left(2280\left(\frac{196}{10^7}\right)^{5/7}\right)^i<0.99^i=o(1/n)$, which finishes the proof.
			\item We again use a union bound, now over all disjoint $S$ and $T$ of sizes $s=\frac{21n}{10^7}$ and $t=\frac{175n}{10^7}$ respectively (note that it suffices to consider sets $T$ of size exactly $t$). We get that the probability of a bad outcome is at most
			\begin{align*}
				\binom{n}{s}\binom{n}{t}\binom{ts}{95s/7}p^{95s/7}&\leq \left(\frac{10^7ne}{21n}\right)^{\frac{21n}{10^7}}
				\left(\frac{10^7ne}{175n}\right)^{\frac{175n}{10^7}}
				\left(\frac{7\cdot 175ne}{95\cdot 10^7}\right)^{95s/7}\left(\frac{64}{n}\right)^{95s/7}\\
				&= \left[\left(\frac{10^7e}{21}\right)^{\frac{21}{10^7}}
				\left(\frac{10^7e}{175}\right)^{\frac{175}{10^7}}
				\left(\frac{1225e\cdot64}{95\cdot 10^7}\right)^{\frac{95\cdot 21}{7\cdot 10^7}}\right]^n<(1-10^{-7})^n.
			\end{align*}
			\item These are standard facts, so we omit the proofs.
		\end{enumerate}
	\end{proof}

	
	\lateproof{Theorem~\ref{thm:ramsey}}
	We will show that for large enough $n$ we have that $\hat{R}_{\ind}(P_{7n/10^7})\leq (1+o(1))32n$, which gives $\hat{R}_{\ind}(P_{n})\leq 5\cdot 10^{7}n$.
	
	For large enough $n$, let $G$ be a fixed graph on $n$ vertices which satisfies all the properties given by Lemma~\ref{auxiliary}. Let $\ell=\frac{7n}{10^7}$, $s_1=3\ell$, and $s_2=24\ell$; these are the parameters which we will use when applying Theorem~\ref{thm:local considerations}.
	
	Consider an arbitrary 2-colouring of $G$ and let $G_1$ be the subgraph induced by the majority colour (and containing no isolated vertices); note that $G_1$ is of order at most
	$(1-\varepsilon)n$ for some fixed $\varepsilon>0$, and has at least $(1-o(1))16n$ edges.  
	Let $G'$ be the graph obtained from $G_1$ by successively removing vertices of degree at most $16$, for as long as there are such vertices. $G'$ is not empty, as otherwise $G_1$ contains at most $(1-\varepsilon)16n$ edges.
	Furthermore, we have that $|E(G')|\ge \delta(G')|V(G')|/2\geq \frac{17}{2}|V(G')|$, so by Property 1 of Lemma~\ref{auxiliary} we have that $|V(G')|>\frac{196n}{10^7}=\ell+s_1+s_2$.
	
	We will apply Theorem~\ref{thm:local considerations} to the graphs $G'$ and $G[V(G')]$. Notice that Property 1 from Lemma~\ref{auxiliary} translates directly to the first condition of Theorem~\ref{thm:local considerations}; let us now show that the second condition is also satisfied.
	Suppose towards a contradiction that there is a set $S\subseteq V(G')$ such that $|S|=s_1=\frac{21n}{10^7}$ and $|N_{G'}(S)|< s_2+\ell=\frac{175n}{10^7}$. Note that $$e_G(S,N_{G'}(S))\geq \delta(G')|S|-2\cdot e_{G'}(S)\geq 17|S|-2\frac{12}{7}|S|\geq \frac{95}{7}|S|,$$ where the second inequality follows from Property 1; this gives a contradiction with Property 2 applied to the sets $S$ and $N_{G'}(S)$. So we can apply Theorem~\ref{thm:local considerations}, and find the required monochromatic path in $G'$, which is induced in $G$. Since $G$ has at most $(1+o(1))32n$ edges, and we can always find an induced path of length $\ell=\frac{7n}{10^7}$ in any 2-colouring of $E(G)$, this gives the required bound on $\hat{R}_{\ind}(P_{7n/10^7})$.
	\endproof{}

	\subsection{The multicolour result}
	In this section we again show an auxiliary lemma about random graphs with certain parameters, which is then used to prove Theorem~\ref{thm:multicolour}.
	
	\begin{lemma}\label{auxiliary-multicolour}
		There exists $c>100$ such that for all $k\ge e^{13}$ the following holds whp for $G\sim G(kn,\frac{c\log k}{n})$. 
		\begin{enumerate}
			\item Every vertex set $S$ of size at most $\frac{2n}{c^3k\log^2k}$ spans less than $\frac{2\log k}{\log k+2}|S|$ edges.
			\item Every two vertex sets $S,T$ of sizes $|S|=\frac{n}{c^3k\log^3k}$ and $|T|\le \frac{2n}{c^3k\log^2 k}$ satisfy $e(S,T)<8|S|\log k$.
			\item $G$ has $(\frac{1}{2}+o(1))cnk^2\log k$ edges.
		\end{enumerate}
	\end{lemma}
	
	\begin{proof}
		We choose $c$ to be a large enough constant not depending on $k$. 
		\begin{enumerate}
			\item 
			Let $t=\frac{2n}{c^3k\log^2k}$ and let $\alpha=\frac{2\log k}{\log k+2}>\frac{5}{3}$ and note that $\alpha<2$, and let $p=\frac{c\log k}{n}$. As before, we bound the probability of the existence of a set $S$ of size at most $t$ which spans at least $\alpha|S|$ edges:
			\begin{align*}
				&\sum_{i\le t }\binom{kn}{i}\binom{\binom{i}{2}}{\alpha i}\cdot p^{\alpha i}\le
				\sum_{i\le t}\left(\frac{ekn}{i}\right)^i\cdot \left(\frac{eip}{2\alpha}\right)^{\alpha i}
				=\sum_{i\le t}\left[\frac{ekn}{i}\cdot\left(\frac{cei\log k }{2\alpha n}\right)^{\alpha}\right]^i\\
				&\leq \sum_{i\le t}\left[10c^2k\log^2k \left(\frac{i}{n}\right)^{\alpha-1}\right]^i
				\leq\sum_{i\le t}\left[10c^2k\log^2k \left(\frac{2}{c^3k\log^2 k}\right)^{1-\frac{4}{\log k}}\right]^i
				\,.
			\end{align*}
			The part of the sum for $i<\sqrt{n}$ we crudely bound by $\sum_{i\le \sqrt{n}}\left(\frac{1}{n^{1/4}}\right)^i\rightarrow 0$, by looking at the penultimate sum above and using the bound on $i$ and $\alpha> \frac{5}{3}$. For the remaining part of the sum, the $i$-th summand is bounded by $\left(\frac{1}{2}\right)^i=o(1/n)$ since $1-\frac{4}{\log k}>2/3$ and $c$ is large enough, which finishes the proof.
			\item Using a union bound over all disjoint $S$ and $T$ of sizes $s=\frac{n}{c^3k\log^3k}$ and $t=\frac{2n}{c^3k\log^2 k}$ respectively, we get that the probability of a bad outcome is at most
			\begin{align*}
				\binom{kn}{s}\binom{kn}{t}\binom{ts}{8s\log k}p^{8s\log k}&\leq \left(\frac{ekn}{t}\right)^{2t}
				\left(\frac{et}{8\log k}\right)^{4t}
				\left(\frac{c\log k}{n}\right)^{4t}=\left(\frac{c^2e^3kt}{2^6n}\right)^{2t}\rightarrow 0.
			\end{align*}
			\item This is a standard fact, so we omit the proof.
		\end{enumerate}
	\end{proof}
	
	\lateproof{Theorem~\ref{thm:multicolour}}
	Obviously, we can assume that $k$ is large enough, so let $k\ge e^{13}$.
	The proof will follow from the previous lemma and Theorem~\ref{thm:local considerations}, along the lines of the proof of Theorem~\ref{thm:ramsey}, by using the parameters $\ell=s_1=\frac{n}{c^3k\log^3k}$ and $s_2=\frac{n}{c^3k\log^2k}$. Indeed, let $c$ be given by Lemma~\ref{auxiliary-multicolour}, now fix any $k\ge e^{13}$ and let $n$ be large enough such that there exists a graph $G$ on $kn$ vertices which has the three properties given by the previous lemma.
	
	Fix any $k$-colouring of the edges of $G$, and let $G_1$ be the graph induced by the densest colour class. Notice that $G_1$ has at least $ckn\log k/4$ edges, so the average degree in $G_1$ is at least $d=c\log k/2$. Now we obtain the graph $G'$
	from $G_1$ by successively removing all vertices of degree at most $d/4$ until there are none. It is easy to show, by using the first property from Lemma~\ref{auxiliary-multicolour}, that $G'$ has at least $s_1+s_2+\ell$ vertices. 
	As in the proof of Theorem~\ref{thm:ramsey}, we want to show that the two conditions of Theorem~\ref{thm:local considerations} hold for the graphs $G'$ and $G[V(G')]$.
	
	The first one directly follows from the first property of $G$.
	For the second one, let $S\subseteq V(G')$ with $|S|=s_1$, and suppose $|N_{G'}(S)|<s_2+\ell<2s_2$. We reach a contradiction again as in the proof of Theorem~\ref{thm:ramsey}, by using the first and second property of $G$, and the minimum degree condition on $G'$. Hence we can apply Theorem~\ref{thm:local considerations}, and get an induced path in $G'$ of length $\ell$, which completes the proof. 
	\endproof

	\providecommand{\bysame}{\leavevmode\hbox to3em{\hrulefill}\thinspace}
	\providecommand{\MR}{\relax\ifhmode\unskip\space\fi MR }
	\providecommand{\MRhref}[2]{%
		\href{http://www.ams.org/mathscinet-getitem?mr=#1}{#2}
	}
	\providecommand{\href}[2]{#2}

\end{document}